\documentclass[12pt]{amsart}

\usepackage{amssymb}

\newtheorem{theorem}{Theorem}[section]
\newtheorem{proposition}[theorem]{Proposition}
\newtheorem{lemma}[theorem]{Lemma}
\newtheorem{corollary}[theorem]{Corollary}
\newtheorem{definition}[theorem]{Definition}
\newtheorem{remark}[theorem]{Remark}

\numberwithin{equation}{section}

\begin{document}

\baselineskip=15pt

\title[Principal bundles over a real algebraic curve]{Principal
bundles over a real algebraic curve}

\author[I. Biswas]{Indranil Biswas}

\address{School of Mathematics, Tata Institute of Fundamental
Research, Homi Bhabha Road, Bombay 400005, India}

\email{indranil@math.tifr.res.in}

\author[J. Hurtubise]{Jacques Hurtubise}

\address{Department of Mathematics, McGill University, Burnside
Hall, 805 Sherbrooke St. W., Montreal, Que. H3A 2K6, Canada}

\email{jacques.hurtubise@mcgill.ca}

\subjclass[2000]{53C07, 14H60, 14P99}

\keywords{Anti-holomorphic involution, principal bundle,
stable bundle, Einstein--Hermitian connection}

\date{}

\begin{abstract}
Let $X$ be a compact connected Riemann surface equipped with
an anti-holomorphic involution $\sigma$. Let $G$ be a
connected complex reductive affine algebraic group, and let
$\sigma_G$ be a real form of $G$.
We consider holomorphic principal $G$--bundles on $X$
satisfying compatibility conditions with respect to
$\sigma$ and $\sigma_G$. We prove that the points defined over
$\mathbb R$ of the smooth locus of a moduli space of principal 
$G$--bundles on $X$ are precisely these objects, under the
assumption that ${\rm genus}(X)\, \geq\, 3$. Stable,
semistable and polystable
bundles are defined in this context. Relationship between
any of these properties and the corresponding property of
the underlying holomorphic principal $G$--bundle is
explored. A bijective correspondence between unitary
representations and polystable objects is established.
\end{abstract}

\maketitle

\section{Introduction}

Let $X$ be a compact connected Riemann surface, and let
$$\sigma\, :\, X\, \longrightarrow\, X$$ be an anti-holomorphic
involution. Let $G$ be a connected complex reductive linear algebraic
group equipped with an anti-holomorphic involutive automorphism
$\sigma_G$. The subgroup of the center of $G$ that is fixed by
$\sigma$ will be denoted by $Z_{\mathbb R}$.

For a holomorphic principal $G$--bundle $E_G$ on $X$, let
$\overline{E}_G$ be the corresponding $C^\infty$ principal
$G$--bundle obtained by twisting the action of $G$
on $E_G$ using $\sigma_G$. The pullback $\sigma^*
\overline{E}_G$ has a natural holomorphic structure.

A pseudo-real principal $G$--bundle on $X$ is defined to be a pair
of the form $(E_G\, ,\rho)$, where $E_G$ is a holomorphic
principal $G$--bundle on $X$, and
$$
\rho\, :\, E_G\, \longrightarrow\, \sigma^*\overline{E}_G
$$
is a holomorphic isomorphism of principal $G$--bundles satisfying
the condition that there is an element $c\, \in\, Z_{\mathbb R}$
such that the composition
$$
E_G\, \stackrel{\rho}{\longrightarrow}\, \sigma^*\overline{E}_G
\, \stackrel{\sigma^*\overline{\rho}}{\longrightarrow}\,
\sigma^*\overline{\sigma^*\overline{E}}_G\,=\,
\sigma^*\sigma^*\overline{\overline{E}}_G\, =\, E_G
$$
coincides with the automorphism of $E_G$ given by $c$. (The details
are in Section \ref{sec2.1}.) If $c\, =\, e$,
then $(E_G\, ,\rho)$ is called a real principal $G$--bundle.

We define semistable, stable and polystable
pseudo-real principal $G$--bundles. The following is proved (see 
Proposition \ref{prop1} and Proposition \ref{prop2}):

\begin{proposition}\label{prop0}
A pseudo-real principal $G$--bundle $(E_G\, ,\rho)$ is
semistable (respectively, polystable) if and only if $E_G$ is
semistable (respectively, polystable).
\end{proposition}

A comment on the definition of (semi)stability is in order.
Ramanathan and Behrend defined (semi)stable principal bundles
\cite{Ra}, \cite{Be}.
It seems to be a common belief that their definitions are
equivalent. When the base field is algebraically closed, it is
easy to see that the two definitions are equivalent. But here
we are working over the base field $\mathbb R$. In turns out that
in this case their definitions differ. The definition of
Behrend works better. (See Section \ref{sec-2.3} for the details.)

Fix a maximal compact subgroup $K\, \subset\, G$ such that
$\sigma_G(K)\,=\, K$ (we show that such a subgroup exists). Let
$$
\widehat{K} \,:=\, K\rtimes ({\mathbb Z}/2{\mathbb Z})
$$
be the semi-direct product given by the involution $\sigma_G$. Fix
a point $x\, \in\, X$ such that $\sigma(x)\, \not=\, x$. Let
$\Gamma$ be the space of all homotopy classes of paths on $X$
starting from $x$ and ending in either $x$ or $\sigma(x)$.
This set $\Gamma$ has a natural structure of a group.
Let ${\rm Hom}'(\Gamma\, , \widehat{K})$ be the space of
all homomorphisms $\varphi\, :\, \Gamma\, \longrightarrow\,
\widehat{K}$ that fit in the commutative diagram
$$
\begin{matrix}
0 & \longrightarrow & \pi_1(X,x)& \longrightarrow & \Gamma &
\longrightarrow & {\mathbb Z}/2{\mathbb Z} &
\longrightarrow & 0\\
&& \Big\downarrow && ~\, \Big\downarrow \varphi && \Vert \\
0 & \longrightarrow & K & \longrightarrow & \widehat{K} &
\longrightarrow & {\mathbb Z}/2{\mathbb Z} &
\longrightarrow & 0
\end{matrix}
$$

We prove the following theorem (see Theorem \ref{thm1}):

\begin{theorem}\label{thm0}
Let $G$ be semisimple.
Isomorphism classes of polystable real principal $G$--bundles
on $X$ are in bijective
correspondence with ${\rm Hom}'(\Gamma\, , \widehat{K})/K$.
\end{theorem}

Let us assume that the involutions $\sigma$ and $\sigma_G$ are such that
$\sigma^*\overline{E}_G$ and $E_G$ are of the same topological type. Let 
${\mathcal M}_X(G)$ be the moduli space of 
stable principal $G$--bundles on $X$ of the given topological type.
The smooth locus of ${\mathcal M}_X(G)$ will be denoted by
${\mathcal M}^s_X(G)$. We have an anti-holomorphic involution
$$
\eta\, :\, {\mathcal M}^s_X(G)\, \longrightarrow\,
{\mathcal M}^s_X(G)
$$
defined by $E_G\, \longmapsto\, \sigma^*\overline{E}_G$.

We prove the following (see Theorem \ref{thm2}):

\begin{theorem}\label{thm-1}
Assume that ${\rm genus}(X)\, \geq\, 3$. Take any
principal $G$--bundle $E_G\, \in\, {\mathcal M}^s_X(G)$. 
Then this $E_G$ is fixed by the involution $\eta$ if and
only if $E_G$ admits a pseudo-real structure.
\end{theorem}

\section{Principal bundles and semistability over a curve
defined over $\mathbb R$}

\subsection{Pseudo-real principal bundles}\label{sec2.1}

Let $G$ be a connected complex algebraic group. Let $\overline{G}$
be the complex algebraic group given by $G$ using the automorphism
of the base field $\mathbb C$ defined by $\lambda\,
\longmapsto\, \overline{\lambda}$. Let
\begin{equation}\label{e1}
\sigma_G\, :\, G\, \longrightarrow\, G
\end{equation}
be an anti-holomorphic involutive homomorphism,
meaning $\sigma_G\circ \sigma_G\, =\, \text{Id}_G$. So
$\sigma_G$ is an algebraic isomorphism of $G$ with
$\overline{G}$. Let
\begin{equation}\label{e2}
G_{\mathbb R}\, \subset\, G
\end{equation}
be the fixed point set for $\sigma_G$; it is a real
analytic group. The center of $G$ will be denoted by $Z$.
The intersection
\begin{equation}\label{e3}
Z_{\mathbb R}\, :=\, Z\bigcap G_{\mathbb R}
\end{equation}
coincides with the center of $G_{\mathbb R}$. Indeed, this follows
immediately from the fact that $G_{\mathbb R}$ is Zariski
dense in $G$.

Let $X$ be a compact connected Riemann surface equipped
with an anti--holomorphic involution
\begin{equation}\label{e4}
\sigma\, :\, X\, \longrightarrow\, X\, .
\end{equation}

Let $E_G\, \longrightarrow\, X$ be a holomorphic principal
$G$--bundle over $X$. Let
$$
\overline{E}_G\, =\, E_G(\sigma_G) \, :=\, E_G\times^{\sigma_G} G
\, \longrightarrow\, X
$$
be the $C^\infty$ principal $G$--bundle over $X$ obtained by
extending the structure group of $E_G$ using the homomorphism
$\sigma_G$ in \eqref{e1}. So $\overline{E}_G$ is a holomorphic
principal $\overline{G}$--bundle over $X$.

\begin{remark}\label{rem1}
{\rm The total space of $\overline{E}_G$
is canonically identified with the total space of $E_G$.
This identification sends any $z\, \in\, E_G$ to the element
of $\overline{E}_G$ given by $(z\, ,e)\, \in\, E_G\times G$;
recall that the total space of $\overline{E}_G$ is the quotient
of $E_G\times G$ where two elements $(z_1\, ,g_1)$ and
$(z_2\, ,g_2)$ of $E_G\times G$ are identified in $\overline{E}_G$
if and only if
there is an element $h\, \in\, G$ such that $z_2\,=\, z_1h$
and $g_2\,=\, \sigma_G(h^{-1})g_1$. The inverse map
$\overline{E}_G\, \longrightarrow\, E_G$ is obtained from
the map $E_G\times G\, \longrightarrow\, E_G$ defined by
$(z\, ,g)\, \longmapsto\, z\sigma^{-1}_G(g)$.}
\end{remark}

\begin{remark}\label{rem0}
{\rm We note that there is no natural holomorphic structure
on the principal $G$--bundle $\overline{E}_G$ over $X$
because $\sigma_G$ is not holomorphic (although it is a
holomorphic principal $\overline G$--bundle).
On the other hand, since both $\sigma_G$ and
$\sigma$ (defined in \eqref{e4}) are anti--holomorphic, the
pulled back principal $G$--bundle $\sigma^*\overline{E}_G$ has
a natural holomorphic structure given by the holomorphic structure
on $E_G$. To describe the holomorphic structure on 
$\sigma^*\overline{E}_G$, first note that the total space of
$\sigma^*\overline{E}_G$ is naturally identified with the
total space of $\overline{E}_G$; one can keep the set (for
the total space) fixed, and
simply modify the projection to $X$ by $\sigma$. In Remark 
\ref{rem1} we saw
that the total space of $\overline{E}_G$ is identified with the
total space of $E_G$;
indeed it is the same set, but with the action 
of $G$ twisted by $\sigma_G$. Combining these two identifications,
we obtain a natural identification of $\sigma^*\overline{E}_G$ with 
$E_G$; this identification does not commute with projection to $X$,
however, but must be intertwined with $\sigma$, i.e., the
identification is a lift of $\sigma$. With this identification, the 
holomorphic structure on 
$\sigma^*\overline{E}_G$ is uniquely determined by the condition
that the identification between the total spaces of
$\sigma^*\overline{E}_G$ and $E_G$ is anti-holomorphic.}
\end{remark}

As both $\sigma$ and $\sigma_G$ are involutions, the principal
$G$--bundles $\sigma^*\sigma^*E_G$ and $\overline{\overline{E}}_G$
are identified with $E_G$. We also note that $\sigma^*\overline{E}_G$
is identified with $\overline{\sigma^*E}_G$.

For any $z\, \in\, Z$ (see \eqref{e3}), and any
holomorphic principal $G$--bundle $F_G$, the map $F_G\, 
\longrightarrow\, F_G$ defined by $y\, \longmapsto\, yz$ is a
holomorphic isomorphism of principal $G$--bundles.
For any $z\, \in\, Z\setminus\{e\}$, the holomorphic
automorphism of $F_G$ given by $z$ is nontrivial. Therefore,
we have
\begin{equation}\label{cZ}
Z\, \subset\, \text{Aut}(F_G)\, ,
\end{equation}
where $\text{Aut}(F_G)$ is the group of all holomorphic automorphisms
of the principal $G$--bundle $F_G$ over the identity map of $X$.

\begin{definition}\label{def1}
{\rm A} pseudo-real {\rm principal $G$--bundle on $X$ is a pair
of the form $(E_G\, ,\rho)$, where $E_G\, \longrightarrow\, X$ is a 
holomorphic principal $G$--bundle, and}
$$
\rho\, :\, E_G\, \longrightarrow\, \sigma^*\overline{E}_G
$$
{\rm is a holomorphic isomorphism of principal $G$--bundles satisfying
the condition that there is an element $c\, \in\, Z_{\mathbb R}$
such that the composition}
$$
E_G\, \stackrel{\rho}{\longrightarrow}\, \sigma^*\overline{E}_G
\, \stackrel{\sigma^*\overline{\rho}}{\longrightarrow}\,
\sigma^*\overline{\sigma^*\overline{E}}_G\,=\,
\sigma^*\sigma^*\overline{\overline{E}}_G\, =\, E_G
$$
{\rm coincides with the automorphism of $E_G$ given by $c$.}
\end{definition}

Some clarifications on the above definition are in order.
The morphism
$$
\overline{\rho}\, :\,\overline{E}_G\, \longrightarrow\,
\overline{\sigma^*\overline{E}}_G
$$
in Definition \ref{def1} is the one given by $\rho$ using the
natural identifications of $E_G$ and $\sigma^*\overline{E}_G$
with $\overline{E}_G$ and $\overline{\sigma^*\overline{E}}_G$
respectively (see Remark \ref{rem1}).
Since $\sigma^*\overline{E}_G\,=\, \overline{\sigma^*E}_G$, 
it follows immediately that $\sigma^*\overline{\sigma^*\overline{E}}_G
\,=\, \sigma^*\sigma^*\overline{\overline{E}}_G$.
The element $c$ in Definition \ref{def1} is unique because
$Z$ is a subgroup of $\text{Aut}(E_G)$ (see \eqref{cZ}).

An \textit{isomorphism} between
two pseudo-real principal $G$--bundles $(E_G\, ,\rho)$ and
$(F_G\, ,\delta)$ is a holomorphic isomorphism of principal
$G$--bundles
$$
\mu\, :\, E_G\, \longrightarrow\, F_G
$$
such that the following diagram commutes:
$$
\begin{matrix}
E_G & \stackrel{\rho}{\longrightarrow} & \sigma^*\overline{E}_G\\
~\Big\downarrow\mu && ~\,~\,~
\,\,\,\Big\downarrow \sigma^*\overline{\mu}\\
F_G & \stackrel{\delta}{\longrightarrow} & \sigma^*\overline{F}_G
\end{matrix}
$$
where $\sigma^*\overline{\mu}$ is the holomorphic isomorphism of 
principal $G$--bundles given by $\mu$; the map $\sigma^*\overline{\mu}$
coincides with $\mu$ using the identification of the total
spaces of $E_G$ and $F_G$ with $\sigma^*\overline{E}_G$
and $\sigma^*\overline{F}_G$ respectively (see Remark \ref{rem0}).

\begin{definition}\label{def2}
{\rm A} real {\rm principal $G$--bundle on $X$ is a pseudo-real
principal $G$--bundle $(E_G\, ,\rho)$ such that the composition}
$$
E_G\, \stackrel{\rho}{\longrightarrow}\, \sigma^*\overline{E}_G
\, \stackrel{\sigma^*\overline{\rho}}{\longrightarrow}\,
\sigma^*\overline{\sigma^*\overline{E}}_G\,=\,
\sigma^*\sigma^*\overline{\overline{E}}_G\, =\, E_G
$$
{\rm is the identity automorphism.}
\end{definition}

Therefore, a pseudo-real principal $G$--bundle $(E_G\, ,\rho)$ as in
Definition \ref{def1} is real if and only if $c\,=\, e$.

An alternate way of viewing these structures is as anti-holomorphic
lifts
\begin{equation}\label{sl}
\begin{matrix} E_G &\buildrel{\widetilde \sigma}\over 
{\longrightarrow}& E_G\\
\Big\downarrow&&\Big\downarrow\\ X&\buildrel{ \sigma}\over 
{\longrightarrow}& X\end{matrix}
\end{equation}
such that ${\widetilde\sigma} (z\cdot g)\,=\, 
{\widetilde\sigma} (z)\cdot\sigma_G (g)$ and the composition
${\widetilde\sigma}\circ{\widetilde\sigma}$ is
fiberwise multiplication by $c$;
see Remark \ref{rem0}.

Consider $\widetilde \sigma$ in \eqref{sl}.
We can modify the lift $\widetilde \sigma$ by the action of an
element $a$ of $Z$ as follows: $\widetilde{\sigma}'(z) \,:=\,
\widetilde\sigma (z) \cdot a$.
This then gives the constraint $a\sigma_G(a) \,\in\, Z_{\mathbb R}$,
and the element $c$ gets replaced by $a\sigma_G(a)c\,=\,
ca\sigma_G(a)$.

In particular, we can take $a$ lying in $Z_{\mathbb R}$, and the
composition gets altered by $a^2$. Therefore, if $c$ lies in
$Z_{\mathbb R}^2\, :=\, \{z^2\, \mid\, z \,\in\, Z_{\mathbb R}\}$,
or more generally is of the form $\sigma_G(a)a$, we
can normalize our pseudo-real structure to a real one.

Let $Z_{\mathbb R}(2)$ be the group of points of $Z_{\mathbb R}$
of order two. The natural homomorphism $Z_{\mathbb
R}(2)\, \longrightarrow\, Z_{\mathbb R}/Z_{\mathbb R}^2$ is
surjective. So one can assume, as we will from now on, that the
element $c$ in Definition \ref{def1}
is of order two. In particular, it lies in any
maximal compact subgroup of $G$.

\subsection{Stable and semistable principal bundles}

Henceforth, we assume that the group $G$ is reductive.

Let $(E_G\, ,\rho)$ be a pseudo-real principal $G$--bundle
over $X$. Let
$$
\text{Ad}(E_G)\, :=\, E_G\times^G G\, \longrightarrow\, X
$$
be the group-scheme over $X$ associated to $E_G$ for the adjoint
action of $G$ on itself. So $\text{Aut}(E_G)$ in \eqref{cZ}
is the space of all holomorphic sections of $\text{Ad}(E_G)$.
Consider the $C^\infty$ principal $G$--bundle $\overline{E}_G$.
Since it is, by definition, the extension of structure group of 
$E_G$ by the isomorphism $\sigma_G$ in \eqref{e1}, the
homomorphism $\sigma_G$ induces a $C^\infty$ isomorphism
\begin{equation}\label{alpha}
\alpha\, :\, \text{Ad}(E_G)\, \longrightarrow\,
\text{Ad}(\overline{E}_G)\, ,
\end{equation}
where $\text{Ad}(\overline{E}_G)$ is the adjoint bundle
for $\overline{E}_G$. More precise, both $\text{Ad}(E_G)$
and $\text{Ad}(\overline{E}_G)$ are quotients of $E_G\times G$;
the map $\alpha$ in \eqref{alpha} is the descent of the
self-map $\text{Id}_{E_G}\times \sigma_G$ of $E_G\times G$.
Note that for each point $x\, \in\, X$, the restriction
$$
\alpha(x) \, :\, \text{Ad}(E_G)_x\, \longrightarrow\,
\text{Ad}(\overline{E}_G)_x
$$
is an isomorphism of groups.

The isomorphism $\rho$ produces a holomorphic isomorphism
$$
\text{Ad}(E_G)\, \longrightarrow\,
\text{Ad}(\sigma^* \overline{E}_G)\,=\, \sigma^* \text{Ad}(
\overline{E}_G)\, .
$$
Combining this with $\alpha$ in \eqref{alpha}, we obtain
an anti-holomorphic automorphism $\widetilde{\rho}$ on 
$\text{Ad}(E_G)$ over $\sigma$, meaning $\widetilde{\rho}$ fits
in the commutative diagram
\begin{equation}\label{e5}
\begin{matrix}
\text{Ad}(E_G) &\stackrel{\widetilde{\rho}}{\longrightarrow}&
\text{Ad}(E_G)\\
\Big\downarrow && \Big\downarrow\\
X&\stackrel{\sigma}{\longrightarrow}& X
\end{matrix}
\end{equation}
We note that $\widetilde{\rho}\circ \widetilde{\rho}\,=\,
\text{Id}_{\text{Ad}(E_G)}$ because
the adjoint action of $Z_{\mathbb R}$ on $G$ is the trivial one.

Let
$$
\text{ad}(E_G)\, :=\, E_G\times^G{\mathfrak g}
\, \longrightarrow\, X
$$
be the bundle of Lie algebras over $X$ associated to $E_G$ for the 
adjoint action of $G$ on ${\mathfrak g}\, :=\, \text{Lie}(G)$;
it is called the \textit{adjoint} vector bundle.

A \textit{proper parabolic} subgroup-scheme of $\text{Ad}(E_G)$ is
a Zariski closed analytically locally trivial proper
subgroup-scheme $P\,\subset\,
\text{Ad}(E_G)$ such that $\text{Ad}(E_G)/P$ is compact. For
an analytically locally trivial
subgroup-scheme $P\,\subset\, \text{Ad}(E_G)$, let
${\mathfrak p}\, \subset\, \text{ad}(E_G)$ be the bundle of
Lie subalgebras corresponding to $P$.

\begin{definition}\label{def3}
{\rm A pseudo-real principal $G$--bundle $(E_G\, ,\rho)$
over $X$ is called} semistable {\rm (respectively,} stable{\rm{)}
if for every proper parabolic subgroup-scheme $P\,\subset\, 
\text{Ad}(E_G)$ such that $\widetilde{\rho}(P)\, \subset\, P$,
where $\widetilde{\rho}$ is constructed in \eqref{e5},}
$$
{\rm degree}({\mathfrak p})\, \leq\, 0~\, ~\,~{\rm (respectively,~
{\rm degree}({\mathfrak p})\, <\, 0)}
$$
{\rm (the vector bundle ${\mathfrak p}$ is defined above).}
\end{definition}

The above definition coincides with the one in \cite[page 304,
Definition 8.1]{Be}, but it differs from the definition
of (semi)stability given in \cite{Ra}. The definitions of Behrend
and Ramanathan are equivalent if the base field is
$\mathbb C$. The difference between the definitions in
\cite{Be} and \cite{Ra} of (semi)stable
principal bundles will be explained in Section \ref{sec-2.3}.

For holomorphic vector bundles on $X$ (so the base field
is $\mathbb C$), we will adopt Ramanathan's
definition of (semi)stability. We reiterate that in this
case, this definition is equivalent to the one given in \cite{Be}.

See \cite{Ra}, \cite[page 221, Definition 3.5]{AB} for the definition
of a polystable principal bundle over a compact Riemann surface.

\begin{proposition}\label{prop1}
A pseudo-real principal $G$--bundle $(E_G\, ,\rho)$ over $X$ is
semistable (respectively, stable) if the principal 
$G$--bundle $E_G$ is semistable (respectively, stable).

For a semistable pseudo-real principal $G$--bundle $(E_G\, ,\rho)$, 
the principal $G$--bundle $E_G$ is semistable.

If $(E_G\, ,\rho)$ is a stable pseudo-real principal $G$--bundle
over $X$, then the principal $G$--bundle $E_G$ is polystable.
\end{proposition}

\begin{proof}
Given a parabolic subgroup--scheme $P\, \subset\, \text{Ad}(E_G)$,
there is a parabolic subgroup $Q\, \subset\, G$ and a
holomorphic reduction of
structure group $E_Q\, \subset\, E_G$ such that the subgroup--scheme
$\text{Ad}(E_Q)\, \subset\, \text{Ad}(E_G)$ coincides with $P$
(see the proof of Lemma 2.11 in \cite{AB}). (The parabolic 
subgroup--scheme $P$ defines a conjugacy class of parabolic subgroups
of $G$, and the reduction of
structure group $E_Q$ depends only on the choice of a parabolic
subgroup $Q$ in this conjugacy class.) Since any parabolic 
subgroup--scheme is given by a reduction of
structure group to a parabolic subgroup, it follows immediately that
$(E_G\, ,\rho)$ is semistable (respectively, stable) if the principal
$G$--bundle $E_G$ is semistable (respectively, stable).

To prove that $E_G$ is semistable if $(E_G\, ,\rho)$ is so,
assume that $E_G$ is not semistable. Then the
adjoint vector bundle $\text{ad}(E_G)$ is not semistable
\cite[page 698, Lemma 3]{AAB}. Using the Harder--Narasimhan
filtration of $\text{ad}(E_G)$, we get a parabolic subalgebra
bundle $\mathfrak p$ of $\text{ad}(E_G)$, which in turn gives a 
parabolic subgroup-scheme
\begin{equation}\label{P}
P\,\subset\, \text{Ad}(E_G)
\end{equation}
\cite[page 699, Lemma 4]{AAB}. We recall from \cite{AAB} that the 
Harder--Narasimhan filtration of $\text{ad}(E_G)$ is of the form
\begin{equation}\label{fi}
0\,=\, E_{-\ell-1} \,\subset\, E_{-\ell}\,\subset\,\cdots\,\subset
\, E_{-1}\,\subset\, E_0 \,\subset\, E_1\,\subset\,\cdots
\,\subset\, E_{\ell-1}\,\subset\, E_\ell
\,=\, \text{ad}(E_G)
\end{equation}
for some positive integer $\ell$. The subbundle $\mathfrak p$
mentioned above is $E_0$ in \eqref{fi}.

Let $\beta\, :\, \text{ad}(E_G)\, \longrightarrow\,
\text{ad}(\overline{E}_G)$ be the $C^\infty$ isomorphism
given by $\alpha$ in \eqref{alpha}. Using $\beta$,
any $C^\infty$ subbundle $V\, \subset\, \text{Ad}(E_G)$ produces
a subbundle of $\text{ad}(\sigma^*\overline{E}_G)
\,=\, \sigma^*\text{ad}(\overline{E}_G)$, which we will
denote by $V'$. This subbundle $V'$ is holomorphic if and only
if $V$ is a holomorphic subbundle. (If $h\, :\, M\,
\longrightarrow\, N$ is an anti-holomorphic isomorphism between
complex manifolds, then a $C^\infty$ submanifold $M'\, \subset\,
M$ is complex analytic if and only if $f(M')$ is complex analytic.)

For $i\, \in\, [-\ell-1\,
, \ell]$, let
$$
E'_i\, \subset\, \text{ad}(\sigma^*\overline{E}_G)
$$
be the holomorphic subbundle corresponding to $E_i$ in \eqref{fi}.

Let
\begin{equation}\label{ad}
\rho_{\rm ad}\, :\, \text{ad}(E_G)\, \longrightarrow\,
\text{ad}(\sigma^*\overline{E}_G)
\end{equation}
be the holomorphic isomorphism induced by $\rho$. Note that
$\rho_{\rm ad}$ is an involution because the central element
$c$ in Definition \ref{def1} acts trivially on $\mathfrak g$.
Since $$\text{rank}(E'_i)\,=\, \text{rank}(E_i)~\,\text{~and~}\,~
\text{degree}(E'_i)\,=\, \text{degree}(E_i)\, ,$$ and
$\text{ad}(E_G)$ is holomorphically isomorphic to
$\text{ad}(\sigma^*\overline{E}_G)$, we conclude that
$$
0\,=\, E'_{-\ell-1} \,\subset\, E'_{-\ell}\,\subset\,\cdots
\,\subset\, E'_{-1}\,\subset\, E'_0 \,\subset\, E'_{1}
\,\subset\,\cdots\,\subset\, E'_{\ell-1} \,\subset\, E'_\ell
\,=\, \text{ad}(\sigma^*\overline{E}_G)
$$
is the Harder--Narasimhan filtration of 
$\text{ad}(\sigma^*\overline{E}_G)$. Indeed, for any
$i\, \in\, [-\ell-1\, ,\ell-1]$, the quotient 
$E_{i+1}/E_i$ is the unique subbundle of $\text{ad}(E_G)/E_i$
of maximal rank among the subbundles of maximal slope. From this
it follows that the above filtration of 
$\text{ad}(\sigma^*\overline{E}_G)$
coincides with its Harder--Narasimhan filtration.

Since any holomorphic isomorphism between vector bundles preserves
their Harder--Narasimhan filtrations, we conclude that
$$
\rho_{\rm ad}(E_0)\, =\, E'_0\, ,
$$
where $\rho_{\rm ad}$ is constructed in \eqref{ad}.
This implies that $\widetilde{\rho}(P)\, \subset\, P$,
where $\widetilde{\rho}$ is constructed in \eqref{e5}, and
$P$ is the subgroup-scheme in \eqref{P}. 

We note that
\begin{equation}\label{da}
\text{degree}(\text{ad}(E_G))\,=\, 0\, .
\end{equation}
Indeed, any $G$--invariant nondegenerate symmetric bilinear form on
$\mathfrak g$ produces a nondegenerate symmetric bilinear form on
$\text{ad}(E_G)$. This bilinear form on $\text{ad}(E_G)$
identifies $\text{ad}(E_G)$ with its dual $\text{ad}(E_G)^*$.
In particular, \eqref{da} holds.

We have
$$
\text{degree}(E_0/E_{-1})\,=\, 0
$$
implying that $\text{degree}(E_{-1})\,> \, 0$
(see \cite[page 216]{AB}). Hence in view of \eqref{da} we
conclude that the subbundle $E_0\,=\, {\mathfrak p}
\, \subset\, \text{ad}(E_G)$
violates the inequality in Definition \ref{def3}.
So $(E_G\, ,\rho)$ is not semistable. This proves the second part
of the proposition.

To prove the third part of the proposition,
let $(E_G\, ,\rho)$ be a stable pseudo-real principal $G$--bundle
over $X$. From the second part of the proposition we know that
$E_G$ is semistable. Assume that $E_G$ is not polystable.
Since $E_G$ is semistable but not polystable, the adjoint
vector bundle $\text{ad}(E_G)$ is semistable but not polystable
(see \cite[page 214, Proposition 2.10]{AB} and \cite[page 224,
Corollary 3.8]{AB}). Let
\begin{equation}\label{so}
S\, \subset\, \text{ad}(E_G)
\end{equation}
be the unique maximal polystable subbundle, which
is called \textit{socle}, of $\text{ad}(E_G)$
(see \cite[page 23, Lemma 1.5.5]{HL}).

We noted in the proof of the second part
of the proposition that any holomorphic
subbundle of $\text{ad}(E_G)$ produces a holomorphic subbundle
of $\text{ad}(\sigma^*\overline{E}_G)$. Let
$$
S'\, \subset\, \text{ad}(\sigma^*\overline{E}_G)
$$
be the holomorphic subbundle corresponding to the socle $S$.

We will show that
\begin{equation}\label{s}
\rho_{\rm ad}(S)\, =\, S'\, ,
\end{equation}
where $\rho_{\rm ad}$ is constructed in \eqref{ad}. 
To prove this, first note that $\text{ad}(\sigma^*\overline{E}_G)$
is semistable but not polystable, because $\rho_{\rm ad}$ is an
isomorphism of it with $\text{ad}(E_G)$. The holomorphic
subbundles of $S$ are in bijective correspondence with the
holomorphic subbundles of $S'$ by the earlier described
correspondence between subbundles of $\text{ad}(E_G)$ and
subbundles of
$\text{ad}(\sigma^*\overline{E}_G)$. Since this correspondence
preserves both rank and degree, it follows that $S'$ is the
socle of $\text{ad}(\sigma^*\overline{E}_G)$. Now \eqref{s}
follows from the uniqueness of the socle.

Let $Z_0$ be the connected component, containing the
identity element, of the center $Z$ of $G$.
A holomorphic reduction of structure
group $E_Q\, \subset\, E_G$ to a parabolic subgroup $Q$ of $G$
is called \textit{admissible} if for any character $\chi$ of $Q$
trivial on $Z_0$, the associated line bundle $E_Q\times^\chi \mathbb C$
over $X$
is of degree zero. A character $\chi$ of $Q$ is called \textit{strictly
antidominant} if $\chi$ is trivial on $Z_0$, and the associated
line bundle $G\times^\chi \mathbb C$ on $G/Q$ is ample (we are
using the fact that the natural projection $G\, \longrightarrow\, G/Q$
is a principal $Q$--bundle).

We will recall a construction from \cite{AB}.

We have the \textit{socle filtration} of $\text{ad}(E_G)$
$$
S\, :=\, S_0\, \subset\, S_1\, \subset\,\cdots \, \subset\,
S_{k-1}\, \subset\, S_k\, :=\, \text{ad}(E_G)
$$
where $S_i/S_{i-1}$ is the socle of $\text{ad}(E_G)/S_{i-1}$
for all $i\, \in\, [1\, ,k]$.
Using this socle filtration, we get a holomorphic reduction of
structure group
$$
E_Q\, \subset\, E_G\, ,
$$
where
\begin{itemize}
\item $Q\, \subsetneq\, G$
is maximal among all the proper parabolic subgroups $Q'$ of $G$
such that $E_G$ has an admissible reduction of structure group
$$
E'_{Q'}\, \subset\, E_G
$$
for which the associated principal $L(Q')$--bundle 
$$E_{L(Q')}\, =\, E'_{Q'}/R_u(Q')$$
is polystable, where $L(Q')\, :=\, Q'/R_u(Q')$ is the Levi quotient
of $Q'$, and $R_u(Q')$ is the unique maximal normal unipotent
subgroup of $Q'$ (also called the unipotent radical of $Q'$),

\item $E_Q$ is a holomorphic reduction
of structure group of $E_G$ to $Q$ such that the associated
principal $L(Q)$--bundle is polystable, where $L(Q)$ is the Levi
quotient of $Q$.
\end{itemize}
The pair $(Q\, , E_Q)$ is unique in the following sense:
for any other pair $(Q_1\, , E_{Q_1})$ satisfying
the above conditions, there is some $g\, \in\, G$ such that
$Q_1\, =\, g^{-1}Qg$, and $E_{Q_1}\, =\, E_Qg$.
(See \cite[page 218]{AB}.)

Note that from the above relations between $(Q\, , E_Q)$ and
$(Q_1\, , E_{Q_1})$ it follows that the two subgroup-schemes
$\text{Ad}(E_Q)$ and $\text{Ad}(E_{Q_1})$ of $\text{Ad}(E_G)$
coincide.

{}From \eqref{s} it can be deduced that 
$\widetilde{\rho}(\text{Ad}(E_Q))\, =\, \text{Ad}(E_Q)$, where 
$\widetilde{\rho}$ is constructed in \eqref{e5}. Indeed, this
is evident from the construction of $\text{Ad}(E_Q)$ using
the socle filtration.

Since the reduction $E_Q\, \subset\, E_G$ is admissible, and
the adjoint action of $Z_0$ on $\mathfrak g$ is trivial,
the degree of the line bundle on $X$ associated to the
principal $Q$--bundle $E_Q$ for the
character of $Q$ given by the $Q$--module $\bigwedge^{\rm top}
{\rm Lie}(Q)$ is zero. In other words,
${\rm degree}(\text{ad}(E_Q))\,=\, 0$.
This contradicts the fact that
$(E_G\, ,\rho)$ is a stable pseudo-real principal $G$--bundle.
Therefore, we conclude that $E_G$ is polystable.
\end{proof}

The following is a corollary of Proposition \ref{prop1}:

\begin{corollary}\label{cor-n}
A pseudo-real principal $G$--bundle $(E_G\, ,\rho)$ over $X$ is
semistable if and only if the corresponding adjoint real vector
bundle $({\rm ad}(E_G)\, ,\rho_{\rm ad})$ (defined in
\eqref{ad}) is semistable.
\end{corollary}

\begin{proof}
A pseudo-real principal $G$--bundle $(E_G\, ,\rho)$
over $X$ is semistable if and only if $E_G$ is semistable
(Proposition \ref{prop1}), and we know that $E_G$ is semistable
if and only if ${\rm ad}(E_G)$ is semistable
\cite[page 214, Proposition 2.10]{AB}. But ${\rm ad}(E_G)$ is
semistable if and only if $({\rm ad}(E_G)\, ,
\rho_{\rm ad})$ is semistable (recall that $\rho_{\rm ad}$ is an
involution).
\end{proof}

It may be mentioned that the analog of Corollary \ref{cor-n}
for stable bundles is not true. In fact, there are stable
holomorphic vector bundles such that the corresponding
trace zero endomorphism bundle is not stable
(see \cite[page 2212, Example 2 and Remark 2]{HM}).

\subsection{Ramanathan's definition and Behrend's
definition}\label{sec-2.3}

We will first recall Ramanathan's definition of (semi)stability
in our context.

Let $(E_G\, ,\rho)$ be a pseudo-real principal $G$--bundle
over $X$. It will be called \textit{r-semistable}
(respectively, \textit{r-stable}) if for any proper
parabolic subgroup $Q\, \subset\, G$ such that $\sigma_G(Q)\,=\,
Q$, and for any holomorphic reduction of structure group of $E_G$
$$
E_Q\, \subset\, E_G
$$
to $Q$ with $\rho(E_Q)\,=\, E_Q$ (see the next sentence for
a clarification), and for any strictly antidominant 
character $\chi$
of $Q$, the line bundle $E_Q\times^\chi \mathbb C\, \longrightarrow
\, X$ is of nonnegative (respectively, positive) degree. (The
total space of $\sigma^*\overline{E}_G$ is identified with that
of $E_G$, as shown in Remark \ref{rem0}; using this identification,
$\rho(E_Q)$ is considered as a submanifold of $E_G$.)

For any given holomorphic reduction of structure group $E_Q\,
\subset\, E_G$ to $Q\, \subset\, G$, considering the corresponding
subgroup-scheme
$\text{Ad}(E_Q)\,\subset\, E_G$ we conclude that $(E_G\, ,\rho)$
is r-semistable (respectively, r-stable) if $(E_G\, ,\rho)$
is semistable (respectively, stable).

We will construct a r-stable real principal $G$--bundle
$(E_G\, ,\rho)$ which is not semistable.

Take a pair $(X\, ,\sigma)$ as in \eqref{e4}. Take $G\,=\,
\text{GL}(2,{\mathbb C})$, and let $\sigma_G\,=\,
\sigma_{\text{GL}(2,{\mathbb C})}$ be the anti-holomorphic
involution defined by
\begin{equation}\label{A}
A\, \longmapsto\, (\overline{A}^t)^{-1}\, .
\end{equation}
The real subgroup is then ${\rm U}(2)$, and the real structure
corresponds in vector bundle terms to isomorphisms $V\,
\longrightarrow\, \sigma^*(\overline{V}^*)$ (this will be
elaborated later).

Let $L\, \longrightarrow\, X$ be a holomorphic line bundle
with $\text{degree}(L)\, >\, 0$. Define
$$
M\,:=\, (\sigma^*\overline{L})^* ~\,~\text{\, and\,}~\, ~
V\, :=\, L\oplus M\, .
$$
Both $M$ and $V$ are holomorphic vector bundles.
Note that
$$
\sigma^*\overline{V}^*\,=\, \sigma^*\overline{L}^*\oplus
\sigma^*\overline{M}^*\, =\, M\oplus L\, =\, L\oplus M\, .
$$
Therefore, the identity map of $L\bigoplus M$ produces a
holomorphic isomorphism
\begin{equation}\label{r0}
\rho_0\, :\, V\, \longrightarrow\, \sigma^*\overline{V}^*\, .
\end{equation}
The composition
$$
V\, \stackrel{\rho_0}{\longrightarrow}\, \sigma^*\overline{V}^*
\, \stackrel{\sigma^*(\overline{\rho_0}^*)^{-1}}{\longrightarrow}\,
(\sigma^*\overline{\sigma^*\overline{V}^*})^*\,=\,
\sigma^*\sigma^*\overline{\overline{V}}^{**}\,=\, V
$$
is clearly the identity map of $V$.

The vector bundle $V$ defines a principal $\text{GL}(2,{\mathbb
C})$--bundle
over $X$, which we will denote by $E_{\text{GL}_2}$. We recall
that $E_{\text{GL}_2}$ is the space of all $\mathbb C$--linear
isomorphisms from ${\mathbb C}^2$ to the fibers of $V$. Using 
$\rho_0$ in \eqref{r0}, we will construct a map from
$E_{\text{GL}_2}$ to $\sigma^*\overline{E}_{\text{GL}_2}$.

Take a point $x\, \in\, X$, and take any $\psi\, \in\,
(E_{\text{GL}_2})_x$ in the fiber over $x$. So, $\psi$ is
a $\mathbb C$--linear isomorphism from ${\mathbb C}^2$ to $V_x$.
Let
$$
\overline{\psi}\, :\, \overline{\mathbb C}^2\,\longrightarrow\,
\overline{V}_x
$$
be the $\mathbb C$--linear isomorphism defined by $v\, \longmapsto\,
\overline{\psi(\overline{v})}$. Now we have the isomorphism
\begin{equation}\label{c1}
(\overline{\psi}^*)^{-1}\, :\, (\overline{\mathbb C}^2)^*
\,\longrightarrow\, (\overline{V}_x)^*\, .
\end{equation}
On the other hand, the isomorphism $\rho_0$ in \eqref{r0} produces
an isomorphism
$$
\overline{\rho_{0,x}}\, :\, \overline{V}_x
\, \longrightarrow\,(\overline{\sigma^*\overline{V}^*})_x
\,=\, (V_{\sigma(x)})^*
$$
by sending any $w$ to $\overline{\rho_{0,x}(\overline{w})}$, where
$\rho_{0,x}$ is the restriction of $\rho_0$ to $x$. Consider the
corresponding isomorphism
$$
(\overline{\rho_{0,x}}^*)^{-1}\, :\, (\overline{V}_x)^*
\, \longrightarrow\, (V_{\sigma(x)})^{**}\,=\, V_{\sigma(x)}\, .
$$
Let
\begin{equation}\label{c2}
(\overline{\rho_{0,x}}^*)^{-1}\circ (\overline{\psi}^*)^{-1}\, :\, 
(\overline{\mathbb C}^2)^*\, \longrightarrow\, V_{\sigma(x)}
\end{equation}
be the composition of it with the isomorphism in \eqref{c1}.

Consider the standard inner product on ${\mathbb C}^2$; the
standard basis is an orthonormal one. Note that the
$\mathbb C$--linear isometries
of ${\mathbb C}^2$ are the fixed points of the
involution $\sigma_{\text{GL}(2,{\mathbb C})}$ (see \eqref{A}).
This inner product produces a $\mathbb C$--linear isomorphism
of ${\mathbb C}^2$ with $(\overline{\mathbb C}^2)^*$. Let
\begin{equation}\label{c3}
\widetilde{\psi}\, :\, {\mathbb C}^2\, \longrightarrow\, V_{\sigma(x)}
\end{equation}
be the isomorphism given by the one in \eqref{c2} using this
isomorphism of ${\mathbb C}^2$ with $(\overline{\mathbb C}^2)^*$.

Let
$$
\rho_{\text{GL}}\, :\, E_{\text{GL}_2}\, 
\longrightarrow\,\sigma^*\overline{E}_{\text{GL}_2}
$$
be the isomorphism that sends any $\psi$ to $\widetilde{\psi}$
constructed in \eqref{c3}. Note that the total space of
$\overline{E}_{\text{GL}_2}$ is identified with
that of $E_{\text{GL}_2}$
(the maps are constructed in Remark \ref{rem1}); using this
identification, the isomorphism $\widetilde{\psi}$
in \eqref{c3} is considered
as an element of the fiber $(\overline{E}_{\text{GL}_2})_{\sigma(x)}$.

It is straightforward to check that the pair $(E_{\text{GL}_2}\, ,
\rho_{\text{GL}})$ constructed above is a real principal
$\text{GL}(2,{\mathbb C})$--bundle over $X$.

We will show that $(E_{\text{GL}_2}\, , \rho_{\text{GL}})$ is not
semistable. For that, note that for any $x\, \in\, X$, the fiber
$\text{Ad}(E_{\text{GL}_2})_x$ is $\text{GL}(V_x)$
(the space of all $\mathbb C$--linear 
automorphisms of $V_x$). Let
$$
P\, \subset\, \text{Ad}(E_{\text{GL}_2})
$$
be the subgroup-scheme whose fiber over any $x\, \in\, X$ is
the space of all linear automorphisms of $V_x\,=\, L_x\oplus M_x$
that preserve the line $L_x$. It is straightforward to check that
$\widetilde{\rho}(P)\, \subset\, P$, where $\widetilde{\rho}$
is the diffeomorphism in \eqref{e5}. Let $\mathfrak p$ be the
Lie algebra bundle corresponding to $P$. We have
$$
{\mathfrak p}\,=\, \mathcal{E}nd(L)\oplus \mathcal{E}nd(M)\oplus
\mathcal{H}om(M\, ,L)\, \subset\, \mathcal{E}nd(V)\,=\,
\text{ad}(E_{\text{GL}_2})\, .
$$
Therefore,
$$
\text{degree}({\mathfrak p})\,=\, \text{degree}(\mathcal{H}om(M\, ,
L))\,=\, 2\cdot \text{degree}(L)\, >\, 0\, .
$$
Hence $(E_{\text{GL}_2}\, , \rho_{\text{GL}})$ is not
semistable.

On the contrary, $(E_{\text{GL}_2}\, , \rho_{\text{GL}})$ is
r-stable because there is no proper parabolic subgroup of
$\text{GL}(2,{\mathbb C})$ that is preserved by the
involution $\sigma_{\text{GL}(2,{\mathbb C})}$. Indeed, for any
proper parabolic subgroup $Q$ of $\text{GL}(2,{\mathbb C})$,
the intersection $Q\bigcap \sigma_{\text{GL}(2,{\mathbb C})}(Q)$
is isomorphic to ${\mathbb C}^*\times {\mathbb C}^*$.

\begin{lemma}\label{lem1}
Let $(E_G\, ,\rho)$ be a pseudo-real principal $G$--bundle over $X$
satisfying the condition that there is a point $y\, \in\, E_G$
such that $\rho(y)\, =\, y z$ for some element $z$ in the
center $Z$ of $G$ (using the identification of the
total spaces of $\sigma^*\overline{E}_G$ and $E_G$ (see
Remark \ref{rem0}), the element
$\rho(y)$ of $\sigma^*\overline{E}_G$ is considered as
an element of $E_G$). Then $(E_G\, ,\rho)$ is r-semistable
if and only if the principal $G$--bundle $E_G$ is semistable.
\end{lemma}

\begin{proof}
If $E_G$ is semistable, then it is obvious
that $(E_G\, ,\rho)$ is r-semistable.

To prove the converse, take a point $y\, \in\, E_G$
such that $\rho(y)\, =\, y z$ with $z\, \in\, Z$. Let $x\,\in\, X$
be the image of $y$. Since $\rho(y)\, =\, y z$, it follows that
$\sigma(x)\,=\, x$. Let
$$
f_y\, :\, G\, \longrightarrow\, \text{Ad}(E_G)_x
$$
be the map that sends any $g\, \in\, G$ to the image of $(y\, ,g)$
in $\text{Ad}(E_G)_x$ (recall that $\text{Ad}(E_G)_x$ is a quotient
of $(E_G)_x\times G$). This map $f_y$ is a holomorphic isomorphism of
groups.

Define $f_{\rho(y)}\,:=\, f_{yz}\, :\, G\, \longrightarrow\, 
\text{Ad}(E_G)_x$ by replacing $y$ with $\rho(y)\, =\, y z$ in the
above construction of $f_y$.
The two isomorphisms $f_y$ and $f_{yz}$ differ by the automorphism
of $G$ produced by the adjoint action of $z$. Since $z$ is in the
center of $G$, we conclude that
\begin{equation}\label{i}
f_y\, =\, f_{yz}\, .
\end{equation}
{}From \eqref{i} it follows that
\begin{equation}\label{ii}
f_y(\sigma_G(g))\,=\, \widetilde{\rho}(f_{\rho(y)}(g))
\,=\, \widetilde{\rho}(f_y(g))
\end{equation}
for all $g\, \in\, G$, where $\widetilde{\rho}$ and $\sigma_G$ are
the maps defined in \eqref{e5} and \eqref{e1} respectively; note
that since $\sigma(x)\,=\, x$, the map $\widetilde{\rho}$ sends
$\text{Ad}(E_G)_x$ to itself.

Assume that $E_G$ is not semistable. Let $P\,\subset\,
\text{Ad}(E_G)$ be the parabolic subgroup-scheme in \eqref{P}
constructed from the Harder--Narasimhan filtration of the vector
bundle $\text{ad}(E_G)$. Let
$$
Q\, :=\, f^{-1}_y(P_x) \, \subset\, G
$$
be the parabolic subgroup, where $f_y$ is the isomorphism
constructed above. Since $\widetilde{\rho}(P)\,=\, P$, from \eqref{ii}
we conclude that
\begin{equation}\label{iii}
\sigma_G(Q)\,=\, Q\, .
\end{equation}

Let $E_Q\, \subset\, E_G$ be the holomorphic reduction of structure 
group of $E_G$
to $Q$ constructed using the pair $(P\, , Q)$ (see the proof of
Lemma 2.11 in \cite{AB} for the construction of $E_Q$).
Since $\widetilde{\rho}(P)\,=\, P$, from \eqref{iii} it follows
immediately that $\rho(E_Q)\,=\, E_Q$. Therefore, the reduction
$E_Q\, \subset\, E_G$ establishes
that $(E_G\, ,\rho)$ is not r-semistable.
\end{proof}

Proposition \ref{prop1} and Lemma \ref{lem1} together give the
following corollary:

\begin{corollary}\label{cor1}
Let $(E_G\, ,\rho)$ be a pseudo-real principal $G$--bundle
on $X$
satisfying the condition that there is a point $y\, \in\, E_G$
such that $\rho(y)\, =\, y z$ for some $z\, \in\, Z$. Then
$(E_G\, ,\rho)$ is semistable if and only if it is r-semistable.
\end{corollary}

\begin{remark}
{\rm The assumption in Lemma \ref{lem1} and Corollary
\ref{cor1} that $\rho(y)\, =\, y z$ implies that the
image of $y$ in $X$ is fixed by the involution $\sigma$.
Hence this assumption fails if $\sigma$ does not have
a fixed point.}
\end{remark}

\section{Polystable pseudo-real principal bundles and
representations of the fundamental group}

\subsection{Polystable pseudo-real principal bundles}

As before, $G$ is a connected complex reductive group.

Let $E_G$ be a holomorphic principal $G$--bundle over $X$. Let
$P\,\subset\, \text{Ad}(E_G)$ be a proper parabolic subgroup-scheme.
For each point $x\, \in\, X$, the unipotent radical of the fiber
$P_x$ will be denoted by $R_u(P)_x$; it is the unique maximal
normal unipotent subgroup. We have an analytically
locally trivial subgroup-scheme
$$
R_u(P)\, \subset\, P
$$
whose fiber over any $x\, \in\, X$ is $R_u(P)_x$. The quotient
$P/R_u(P)$ is a group-scheme over $X$.

A \textit{Levi subgroup-scheme} of $P$ is an analytically
locally trivial subgroup-scheme
$L(P)\, \subset\, P$ such that the composition
$$
L(P)\, \hookrightarrow\, P\, \longrightarrow\, P/R_u(P)
$$
is an isomorphism. It should be emphasized that a Levi subgroup-scheme
does not exist in general. In vector bundle terms, the existence of a 
Levi subgroup-scheme corresponds to some extension classes being trivial.

\begin{definition}\label{def4}
{\rm A semistable pseudo-real principal $G$--bundle $(E_G\, ,\rho)$
over $X$ is called} polystable {\rm if either $(E_G\, ,\rho)$ is stable,
or there is a proper parabolic subgroup-scheme $P\,\subset\, 
\text{Ad}(E_G)$, and a Levi subgroup-scheme $L(P)\, \subset\, P$,
such that the following
conditions hold:}
\begin{enumerate}
\item {\rm $\widetilde{\rho}(P)\, \subset\, P$ and
$\widetilde{\rho}(L(P))\, \subset\, L(P)$,
where $\widetilde{\rho}$ is constructed in \eqref{e5}, and}

\item for any proper parabolic subgroup-scheme $P'\, \subset\, L(P)$
with $\widetilde{\rho}(P')\, \subset\, P'$, we have
$$
{\rm degree}({\mathfrak p}')\, <\, 0\, ,
$$
where ${\mathfrak p}'$ is the bundle of Lie algebras corresponding
to $P'$.
\end{enumerate}
\end{definition}

It should be clarified that in the above definition, $P'$ is not
a parabolic subgroup-scheme of $\text{Ad}(E_G)$. The condition that
$P'$ is a parabolic subgroup-scheme of $L(P)$ implies that
the quotient $L(P)/P'$ is compact.

\begin{proposition}\label{prop2}
A pseudo-real principal $G$--bundle $(E_G\, ,\rho)$ is polystable
if and only if the principal $G$--bundle $E_G$ is polystable.
\end{proposition}

\begin{proof}
First assume that the pseudo-real principal $G$--bundle
$(E_G\, ,\rho)$ is polystable. We will show that $E_G$
is polystable.

We begin by constructing a reduction of $E_G$ to a suitable parabolic 
subgroup $Q$.
Fix a point $x_0\, \in\, X$, and also fix a point
$z_0\, \in\, (E_G)_{x_0}$ in the fiber over $x_0$. Let
\begin{equation}\label{phi}
\phi\, :\, G\, \longrightarrow\, \text{Ad}(E_G)_{x_0}
\end{equation}
be the map that sends any $g\, \in\, G$ to the image
of $(z\, ,g)$ in $\text{Ad}(E_G)_{x_0}$; recall that
$\text{Ad}(E_G)_{x_0}$ is a quotient of $(E_G)_{x_0}\times G$.
This map $\phi$ is a holomorphic isomorphism of groups.

If $(E_G\, ,\rho)$ is stable, then $E_G$ is polystable by
Proposition \ref{prop1}. Assume that $(E_G\, ,\rho)$ is not stable.
Take $P$ and $L(P)$ as in Definition \ref{def4}. Let
$$
Q\, :=\, \phi^{-1}(P_{x_0})\, \subset\, G
$$
be the parabolic subgroup. The conjugacy class of the subgroup
$Q$ is independent of the choices of $x_0$ and $z_0$ (meaning
subgroups are conjugate by some element of $G$).
We will first show that $E_G$ admits a natural
holomorphic reduction of structure group to $Q$.

For any $x\, \in\, X$, let
$$
(E_Q)_x\, \in\, (E_G)_x
$$
be the submanifold consisting of all points $z$ such that for all
$q\, \in\, Q$, the image of $(z\, ,q)\, \in\, (E_G)_x\times Q$ in 
$\text{Ad}(E_G)_x$ lies in the subgroup $P_x$ (recall that 
$\text{Ad}(E_G)_x$ is a quotient of $(E_G)_x\times Q$). These
$(E_Q)_x$, $x\, \in\, X$, together form a 
holomorphic sub-fiber bundle $E_Q\, \subset\, E_G$. In fact,
$E_Q$ is a holomorphic reduction of structure group of the
principal $G$--bundle $E_G$ to the subgroup $Q$. This follows from
the fact that the normalizer of $Q$ in $G$ coincides with $Q$.
It should be
mentioned that $\rho(E_Q)$ need not coincide with $E_Q$
(we consider $\rho(E_Q)$ as a submanifold of the total space
of $E_G$ using the identification of the total spaces
of $E_G$ and $\sigma^*\overline{E}_G$ described in Remark \ref{rem0}).

Consider $L(P)\, \subset\, P$ as in Definition
\ref{def4}. Define
\begin{equation}\label{lqf}
L(Q)\, :=\, \phi^{-1}(L(P)_{x_0})\, \subset\, Q\, ,
\end{equation}
where $\phi$ is constructed in \eqref{phi}. It is a Levi
subgroup of $Q$, meaning
$L(Q)$ is a connected reductive subgroup of $Q$ such that
the composition
$$
L(Q)\, \hookrightarrow\, Q\, \longrightarrow\, Q/R_u(Q)
$$
is an isomorphism, where $R_u(Q)$ is the unipotent radical of
$Q$. Any two Levi subgroups of $Q$ are conjugate by some
element of $Q$. We will show that $L(P)$ produces a
holomorphic reduction of structure group of the principal
$Q$--bundle $E_Q$ to the subgroup $L(Q)\, \subset\, Q$
defined in \eqref{lqf}.

Let $Z(L(Q))\, \subset\, L(Q)$ be the connected component of
the center of $L(Q)$ containing the identity element. It is a product
of copies of ${\mathbb C}^*$ because $L(Q)$ is reductive. Let
$$
Z(L(P))\,\subset\, L(P)
$$
be the fiber-wise connected component of the center
containing the identity element, meaning the fiber
$Z(L(P))_x$ for any $x\, \in\, X$ is the connected component,
containing the identity element, of the center of $L(P)_x$.
So $Z(L(P))_x$ is isomorphic to $Z(L(Q))$.

We recall that for any $x\, \in\, X$, the fiber $\text{Ad}(E_Q)_x$
is the group of all $Q$--equivariant self-maps of $(E_Q)_x$. Let
\begin{equation}\label{cs}
{\mathcal S}\, \subset\, E_Q
\end{equation}
be the subset consisting of all $(x\, ,z)$, $x\, \in\, X$ and
$z\, \in\, (E_Q)_x$, such that
$$
\phi(t)(z)\,=\, zt
$$
for all $t\, \in\, Z(L(Q))$,
where $\phi$ is constructed in \eqref{phi}. This ${\mathcal S}$
is a holomorphic reduction of structure group of the
principal $Q$--bundle $E_Q$ to the subgroup $L(Q)$.
It should be emphasized that $\rho({\mathcal S})$ need not
coincide with ${\mathcal S}$. Let
\begin{equation}\label{elq}
E_{L(Q)}\, \subset\, E_Q
\end{equation}
be the principal $L(Q)$--bundle
defined by ${\mathcal S}$ in \eqref{cs}. We note that the
subgroup-scheme
$\text{Ad}(E_{L(Q)})\, \subset\, \text{Ad}(E_Q)$ is identified with 
$L(P)$. Indeed, this
follows immediately from the above construction of $\mathcal S$.

We recall from Definition \ref{def4} that
for any proper parabolic subgroup-scheme $P'\, \subset\, L(P)$
with $\widetilde{\rho}(P')\, \subset\, P'$,
\begin{equation}\label{inq}
{\rm degree}({\mathfrak p}')\, <\, 0\, ,
\end{equation}
where ${\mathfrak p}'$ is the bundle of Lie algebras corresponding
to $P'$. In the proof of Proposition \ref{prop1} we saw that
the vector bundle $\text{ad}(E_G)$ is polystable if $(E_G\, ,
\rho)$ is stable. Repeating verbatim this argument for $E_{L(Q)}$ 
(defined in \eqref{elq}), and using \eqref{inq}, we conclude that the
adjoint vector
bundle $\text{ad}(E_{L(Q)})$ is polystable. Note that in the proof
that $\text{ad}(E_G)$ is polystable if $(E_G\, ,
\rho)$ is stable, the involution $\rho$ is not used; only
$\widetilde\rho$ is used.

We also note that $E_{L(Q)}$ is polystable because
$\text{ad}(E_{L(Q)})$ is polystable \cite[page 224,
Corollary 3.8]{AB}.

Consider the adjoint action of $Z(L(Q))$ on ${\mathfrak g}\, :=\, 
\text{Lie}(G)$. Let
\begin{equation}\label{d}
{\mathfrak g}\, =\, \bigoplus_{i=1}^n V_i
\end{equation}
be the isotypical decomposition of the $Z(L(Q))$--module
${\mathfrak g}$. Note that each subspace $V_i\,\subset\,
\mathfrak g$ is a preserved by the adjoint action of
$L(Q)$. For any $i\, \in\, [1\, ,n]$, let
$$
E_{V_i}\, :=\, E_{L(Q)}\times^{L(Q)} V_i\, \longrightarrow\, X
$$
be the holomorphic vector bundle associated to the principal
$L(Q)$--bundle $E_{L(Q)}$ (constructed in \eqref{elq}) for the 
$L(Q)$--module $V_i$. Since
$E_{L(Q)}$ is polystable, the vector bundle $E_{V_i}$ is polystable
\cite[page 285, Theorem 3.18]{RR} (this theorem of \cite{RR} applies
because $Z(L(Q))$ acts on $V_i$ through a character).

Since $(E_G\, ,\rho)$ is semistable (see Definition \ref{def4}), from
the second part of Proposition \ref{prop1} we know that $E_G$ is 
semistable. Hence the adjoint vector bundle $\text{ad}(E_G)$ is
semistable \cite[page 214, Proposition 2.10]{AB}. From \eqref{d}
we have a decomposition
$$
\text{ad}(E_G)\, =\, \bigoplus_{i=1}^n E_{V_i}\, .
$$
Since $\text{ad}(E_G)$ is semistable, and each $E_{V_i}$ is
polystable, we conclude that $\text{ad}(E_G)$ is polystable.
Hence $E_G$ is polystable \cite[page 224,
Corollary 3.8]{AB}.

To prove the converse, assume that the principal $G$--bundle
$E_G$ is polystable. We will prove that $(E_G\, ,\rho)$ is
polystable.

The adjoint vector bundle $\text{ad}(E_G)$ is polystable
because $E_G$ is polystable \cite[page 224, Corollary 3.8]{AB}.
Therefore, $\text{ad}(E_G)$ is semistable. Hence for
every proper parabolic subgroup-scheme $P\,\subset\,
\text{Ad}(E_G)$ such that $\widetilde{\rho}(P)\, \subset\, P$,
where $\widetilde{\rho}$ is constructed in \eqref{e5}, we have
$$
{\rm degree}({\mathfrak p})\, \leq\, 0
$$
(${\mathfrak p}$ is the bundle of Lie algebras associated to $P$).
If ${\rm degree}({\mathfrak p})\, <\, 0$ for every such $P$,
then $(E_G\, ,\rho)$ is stable (see Definition \ref{def3}), in 
particular, it is polystable in that case.

Assume that
\begin{equation}\label{e}
{\rm degree}({\mathfrak p})\, =\, 0
\end{equation}
for a proper parabolic subgroup-scheme $P\,\subset\,
\text{Ad}(E_G)$ with $\widetilde{\rho}(P)\, \subset\, P$.
We also assume that
the rank of the vector bundle ${\mathfrak p}$ is
smallest among all bundles of Lie algebras ${\mathfrak p}''$
satisfying the conditions
that ${\rm degree}({\mathfrak p}'')\, =\, 0$ and
${\mathfrak p}''$ corresponds to a $\widetilde{\sigma}$--invariant
proper parabolic subgroup-scheme of $\text{Ad}(E_G)$.

Since $E_G$ is polystable, it has an Einstein--Hermitian connection
\cite{Ra}, \cite[page 208, Theorem 0.1]{AB}.
Let $\nabla$ be the connection on $\text{ad}(E_G)$ induced by
an Einstein--Hermitian connection on $E_G$. We need to choose
a K\"ahler form on $X$ and a maximal compact subgroup of $G$ in order
to define an Einstein--Hermitian connection on $E_G$. But the induced 
connection on $\text{ad}(E_G)$
is independent of the choices of K\"ahler form on 
$X$ and maximal compact subgroup of $G$. The connection $\nabla$
on $\text{ad}(E_G)$ is
flat unitary. Note that $\text{degree}(\text{ad}(E_G))\,=\, 0$
(see \eqref{da}).

{}From \eqref{e} it follows immediately that the second fundamental
form of the subbundle ${\mathfrak p}\, \subset\, \text{ad}(E_G)$
for the connection $\nabla$
vanishes identically \cite[page 139]{Ko}. Therefore,
the connection $\nabla$ on $\text{ad}(E_G)$ preserves the subbundle 
$\mathfrak p$.

Fix a maximal compact subgroup
\begin{equation}\label{K}
K\, \subset\, G
\end{equation}
such that $\sigma_G(K)\, =\, K$, where $\sigma_G$ is the involution
in \eqref{e1}. To see that such a subgroup exists, let
$$
\widehat{G}\, :=\, G\rtimes ({\mathbb Z}/2{\mathbb Z})
$$
be the semi-direct product for the involution $\sigma_G$ in
\eqref{e1}. Let $\widehat{K}\, \subset\, \widehat{G}$ be
a maximal compact subgroup. Then, the intersection
$\widehat{K}\bigcap G$ is a maximal compact subgroup of $G$
which is preserved by $\sigma_G$.

Fix a K\"ahler form $\omega$ on $X$ such that
$\sigma^*\omega\,=\, -\omega$, where $\sigma$ is the involution
in \eqref{e4}; such a K\"ahler form is constructed by averaging,
with respect to $\sigma$, a Hermitian structure on $TX$. Let
$$
E_K\, \subset\, E_G
$$
be a $C^\infty$ reduction of structure group of the principal
$G$--bundle $E_G$ giving the Einstein--Hermitian connection on $E_G$.

Fix an inner product $h$ on $\mathfrak g$ such that $h$ is
preserved by the adjoint action of $K$ on $\mathfrak g$ (we note that
such an inner product exists because $K$ is compact).
Since $\text{ad}(E_G)$ is identified with the vector bundle
associated to the principal $K$--bundle $E_K$ for the
adjoint action of $K$ on $\mathfrak g$, from the fact that $h$ is
preserved by the adjoint action of $K$ it follows
immediately that the inner product $h$
on $\mathfrak g$ produces a Hermitian structure
on $\text{ad}(E_G)$. This Hermitian structure
on $\text{ad}(E_G)$ will be denoted by $\widehat{h}$.

Since the connection $\nabla$ on $\text{ad}(E_G)$ is induced by
the Einstein--Hermitian connection on $E_G$ given by $E_K$, it
follows immediately that $\nabla$ coincides with the Chern
connection on $\text{ad}(E_G)$ associated to the Hermitian
structure $\widehat{h}$ on it.

Let
$$
R_n({\mathfrak p})\, \subset\, {\mathfrak p}
$$
be the subbundle defined by the fiberwise nilpotent radicals; in
other words, for each point $x\, \in\, X$, the fiber
$R_n({\mathfrak p})_x$ is the nilpotent radical of the 
parabolic subalgebra ${\mathfrak p}_x$. We note that
$R_n({\mathfrak p})$ is a holomorphic subbundle because the
section of ${\mathcal H}om(\bigwedge^2{\mathfrak p}^*\, ,
{\mathfrak p})$ defined by the Lie algebra structure on the fibers
of $\mathfrak p$ is holomorphic. Let
\begin{equation}
L({\mathfrak p})\, :=\, R_n({\mathfrak p})^\perp \bigcap {\mathfrak p}
\,\subset\, {\mathfrak p}
\end{equation}
be the orthogonal complement, with respect to the Hermitian structure
$\widehat{h}$ constructed above, of $L({\mathfrak p})$ inside
${\mathfrak p}$. We will show that this
$C^\infty$ subbundle $L({\mathfrak p})$
is preserved by the connection $\nabla$.

For each point $x\, \in\, X$, the fiber $L({\mathfrak p})_x$ is
closed under the Lie bracket operation on ${\mathfrak p}_x$; in fact,
$L({\mathfrak p})_x$ is a Levi subalgebra of ${\mathfrak p}_x$, meaning
it is a maximal reductive subalgebra and the composition
$$
L({\mathfrak p})_x\, \hookrightarrow\, {\mathfrak p}_x
\, \longrightarrow\, {\mathfrak p}_x/R_n({\mathfrak p})_x
$$
is an isomorphism. All these follow from the
fact that the inner product $h$ on $\mathfrak g$ is preserved
by the adjoint action of $K$. Any
$G$--invariant nondegenerate symmetric bilinear form on
$\mathfrak g$ produces a holomorphic
nondegenerate symmetric bilinear form on
$\text{ad}(E_G)$; the restriction to $L({\mathfrak p})$ of this
holomorphic bilinear form on $\text{ad}(E_G)$ is nondegenerate. 
Therefore, the $C^\infty$ vector bundle $L({\mathfrak p})$ is
isomorphic to its dual $L({\mathfrak p})^*$.

We have $\text{degree}(L({\mathfrak p}))\, =\, 0$ because
$L({\mathfrak p})^*$ is isomorphic to $L({\mathfrak p})$.
Therefore, from \eqref{e} we conclude that
$$
\text{degree}(R_n({\mathfrak p}))\, =\, 0
$$
(recall that ${\mathfrak p}\,=\, R_n({\mathfrak p})\bigoplus
L({\mathfrak p})$).
This implies that the unitary flat connection $\nabla$ 
preserves the holomorphic subbundle $R_n({\mathfrak p})$
of $\text{ad}(E_G)$ \cite[page 139]{Ko}. We already proved
that ${\mathfrak p}$ is preserved by $\nabla$. Consequently, the 
orthogonal complement of $R_n({\mathfrak p})$ in ${\mathfrak p}$,
namely $L({\mathfrak p})$, is preserved by $\nabla$. In particular,
$L({\mathfrak p})$ is a holomorphic subbundle of ${\mathfrak p}$.

Using the exponential map, the bundle of subalgebras
$L({\mathfrak p})\, \subset\, \mathfrak p$ produces a
Levi subgroup-scheme $L(P)\, \subset\, P$. More precisely,
$L(P)$ is the unique subgroup-scheme of $P$ such that the
corresponding bundle of Lie subalgebras coincides with
$L({\mathfrak p})$.

Since $\sigma_G(K)\, =\, K$, and $\sigma^*\omega\,=\, -\omega$, it
follows that $\widetilde{\rho}(L(P))\, =\, L(P)$. Since
$\mathfrak p$ is of smallest rank among all
Lie algebra bundles of degree zero associated to 
proper parabolic subgroup-schemes of $\text{Ad}(E_G)$
preserved by $\widetilde{\rho}$,
we conclude that for any proper parabolic subgroup-scheme $P'\, 
\subset\, L(P)$
with $\widetilde{\rho}(P')\, \subset\, P'$, the inequality
$$
{\rm degree}({\mathfrak p}')\, <\, 0
$$
holds, where ${\mathfrak p}'$ is the bundle of Lie algebras
corresponding to $P'$. Hence $(E_G\, ,\rho)$ is polystable.
This completes the proof of the proposition.
\end{proof}

The following is an analog of Corollary \ref{cor-n}.

\begin{corollary}\label{cor-n2}
A pseudo-real principal $G$--bundle $(E_G\, ,\rho)$ over $X$ is
polystable if and only if the corresponding adjoint real vector
bundle $({\rm ad}(E_G)\, ,\rho_{\rm ad})$ is polystable.
\end{corollary}

\begin{proof}
A pseudo-real principal $G$--bundle
$(E_G\, ,\rho)$ is polystable if and only if the
principal $G$--bundle $E_G$ is polystable
(Proposition \ref{prop2}), and $E_G$ is polystable
if and only if ${\rm ad}(E_G)$ is polystable
\cite[page 224, Corollary 3.8]{AB}. But ${\rm ad}(E_G)$ is
polystable if and only if $({\rm ad}(E_G)
\, ,\rho_{\rm ad})$ is polystable (Proposition \ref{prop2}).
\end{proof}

\subsection{Homomorphisms from fundamental group to a maximal
compact subgroup}

In this subsection we assume that the group $G$ is semisimple.

As in \eqref{K}, fix a maximal compact subgroup
$$
K\, \subset\, G
$$
such that $\sigma_G(K)\,=\, K$. Let
\begin{equation}\label{hK}
\widehat{K}\, :=\, K\rtimes ({\mathbb Z}/2{\mathbb Z})
\end{equation}
be the semi-direct product for the involution $\sigma_G\vert_K$.
The set $\widehat{K}$ is $K\times ({\mathbb Z}/2{\mathbb Z})$;
the group structure is defined using $\sigma_G\vert_K$.

Fix a point $x\, \in\, X$ such that $\sigma(x)\, \not=\, x$.
Let $\Gamma_1$ denote the homotopy classes of paths on $X$ from
$x$ to $\sigma(x)$. We have a group structure on
\begin{equation}\label{Ga}
\Gamma\, :=\, \pi_1(X,x) \cup \Gamma_1
\end{equation}
defined as follows: for $\gamma_1\, ,\gamma_2\, \in\, \Gamma$,
\begin{itemize}
\item if $\gamma_2\, \in\, \pi_1(X,x)$, then $\gamma_2\gamma_1$
is simply the composition of paths $\gamma_1\circ\gamma_2$, and

\item if $\gamma_2\, \in\, \Gamma_1$, then $\gamma_2\gamma_1$
is the composition $\sigma(\gamma_1)\circ \gamma_2$ of paths.
\end{itemize}
(See \cite{BHH} for more details.)

Composition of paths will be denoted by ``$\circ$''.

The group $\widehat{K}$
in \eqref{hK} and the group $\Gamma$ defined above fit in
the exact sequences
$$
0\, \longrightarrow\, K\, \longrightarrow\, \widehat{K}\, 
\longrightarrow\, {\mathbb Z}/2{\mathbb Z} \, \longrightarrow\, 0
$$
and
$$
0\, \longrightarrow\, \pi_1(X,x)\, \longrightarrow\, \Gamma\,
\longrightarrow\, {\mathbb Z}/2{\mathbb Z} \, \longrightarrow\,
0\, .
$$
Let ${\rm Hom}(\Gamma\, , \widehat{K})$ be the space of all
homomorphisms from $\Gamma$ to the group $\widehat{K}$. Let
$$
{\rm Hom}'(\Gamma\, , \widehat{K})\, \subset\,
{\rm Hom}(\Gamma\, , \widehat{K})
$$
be the subset consisting of all homomorphisms $\varphi$ such that
we have commutative diagram
$$
\begin{matrix}
0 & \longrightarrow & \pi_1(X,x)& \longrightarrow & \Gamma &
\longrightarrow & {\mathbb Z}/2{\mathbb Z} &
\longrightarrow & 0\\
&& \Big\downarrow && ~ \, \Big\downarrow \varphi && \Vert \\
0 & \longrightarrow & K & \longrightarrow & \widehat{K} &
\longrightarrow & {\mathbb Z}/2{\mathbb Z} &
\longrightarrow & 0
\end{matrix}
$$
The normal subgroup $K$ of $\widehat{K}$ has a conjugation action
on ${\rm Hom}'(\Gamma\, , \widehat{K})$.

\begin{theorem}\label{thm1}
Let $G$ be a connected semisimple complex affine algebraic group.
Isomorphism classes of polystable real principal $G$--bundles
on $X$ (see Definition \ref{def2}) are in bijective
correspondence with the quotient ${\rm Hom}'(\Gamma\, , 
\widehat{K})/K$.
\end{theorem}

\begin{proof}
Let $(E_G\, ,\rho)$ be a polystable real principal $G$--bundle
on $X$. So $E_G$ is polystable by Proposition \ref{prop2}.
Therefore, $E_G$ has a flat connection given by a
$C^\infty$ reduction of structure group
$$
E_K\, \subset\, E_G\, ,
$$
where $K$ is the subgroup in \eqref{hK}; see \cite[page 146,
Theorem 7.1]{Ra}, \cite[page 208, Theorem 0.1]{AB}. It is possible to
choose $E_K$ such that $\rho(E_K)\,=\, E_K$ \cite[[Proposition 3.7]{BGH};
we fix such a reduction $E_K$. Let $\nabla$ be the flat connection on $E_K$.

Fix a point
\begin{equation}\label{z0}
z_0\, \in\, (E_K)_x\, .
\end{equation}
Taking parallel translations
of $z_0$, for the flat connection $\nabla$, along loops in $X$
based at $x$ we get a homomorphism
\begin{equation}\label{vpp}
\varphi'\, :\, \pi_1(X,x) \, \longrightarrow\, K\, .
\end{equation}

For any $\gamma\, \in\, \Gamma_1$ (see \eqref{Ga}), consider the
parallel translation of the element $z_0$ in \eqref{z0} along
$\gamma$. Let $z'_0\, \in\, (E_K)_{\sigma (x)}$ be the element
obtained by this parallel transport. Using the identification of
the total space of $E_G$ with the total space of $\sigma^*
\overline{E}_G$ (see Remark \ref{rem0}), the element $z'_0$
gives an element $z''_0\, \in\,
(\sigma^*\overline{E}_G)_x$. Therefore,
$$
\rho^{-1}_x(z''_0)\, \in\, (E_G)_x\, ,
$$
where $\rho_x\,=\, \rho\vert_{(E_G)_x}$.

It can be shown that
$\rho^{-1}_x(z''_0)$ lies in the submanifold $(E_K)_x\, 
\subset\, (E_G)_x$. To prove this, first observe that
$\rho(E_K)\, \subset\, E_G$ is a reduction of structure group
of $E_G$
to $\sigma_G(K)$ giving a Einstein--Hermitian connection on $E_G$
(we have again identified $E_G$ with $\sigma^*\overline{E}_G$ using
Remark \ref{rem0}). Now, since $\rho_G(K)\,=\, K$ and $\rho(E_K)\,=\, E_K$, it follows
that $\rho$ preserves $E_K$. Hence $\rho^{-1}_x(z''_0)\,\in\,
(E_K)_x$.

Let
$$
\gamma'\, \in\, K
$$
be the unique element such that $z_0\sigma_G(\gamma')^{-1}\,=\, \rho^{-1}_x(z''_0)$.

Let
\begin{equation}\label{vppp}
\varphi''\, :\,\Gamma_1 \, \longrightarrow\, K
\end{equation}
be the map defined by $\gamma\, \longmapsto\, \gamma'$, where
$\gamma'$ is constructed above from $\gamma$. Let
\begin{equation}\label{vp}
\varphi\, :\,\Gamma \, \longrightarrow\, \widehat{K}
\end{equation}
be the map defined as follows: $\varphi\vert_{\pi_1(X,x)}\,=\,
\varphi'$, where $\varphi'$ is constructed in \eqref{vpp}, and
$$
\varphi(\gamma)\,=\, (\varphi''(\gamma)\, ,1)\, \in\,
K\times ({\mathbb Z}/2{\mathbb Z})\,=\, \widehat{K}\, ,
$$
where $\varphi''$ is constructed in \eqref{vppp}, and
$1\, \in\, {\mathbb Z}/2{\mathbb Z}$ is the nontrivial
element (recall that the natural identification of $\widehat{K}$
with $K\times ({\mathbb Z}/2{\mathbb Z})$ is only set-theoretic).
It is easy to see that
$$
\varphi\, \in\, {\rm Hom}'(\Gamma\, , \widehat{K})\, .
$$

Conversely, given any
\begin{equation}\label{vp-1}
\varphi\, \in\, {\rm Hom}'(\Gamma\, , 
\widehat{K})\, ,
\end{equation}
we will describe a construction of a polystable real principal 
$G$--bundle on $X$.

Consider the restriction of $\varphi$ to $\pi_1(X,x)$. It
produces a flat principal $K$--bundle $E_K$ on $X$
together with a $K$--equivariant isomorphism
$$
\beta\, :\, K\, \longrightarrow\, (E_K)_x
$$
for the right--translation action of $K$ on itself. Define
\begin{equation}\label{z-1}
z_0\, :=\, \beta(e)\, \in\, (E_K)_x\, ,
\end{equation}
where $e\, \in\, K$ is the identity element. The flat
connection on $E_K$ will be denoted by $\nabla$.

Let $E_G\, :=\, E_K\times^K G$ be the principal $G$--bundle
obtained by extending the structure group of $E_K$ using the
inclusion map of $K$ in $G$. The connection $\nabla$ defines
a holomorphic structure on $E_G$. This holomorphic principal
$G$--bundle is polystable because $\nabla$ is flat.

We will construct an isomorphism of $(E_G)_x$ with
$(E_G)_{\sigma(x)}$. Take any element
$$
z_0g \, \in\, (E_G)_x\, ,
$$
where $g\, \in\, G$, and $z_0$ is the element in \eqref{z-1}.
Take an element $\gamma\, \in\, \Gamma_1$ (see \eqref{Ga}). Let
$$
T_\gamma\, :\, (E_G)_x\, \longrightarrow\, (E_G)_{\sigma(x)}
$$
be the parallel translation along $\gamma$. Let
\begin{equation}\label{rpx}
\rho'_{x,\gamma}\, :\, (E_G)_x\, \longrightarrow\, (E_G)_{\sigma(x)}
\end{equation}
be the map defined by $z_0g\, \longmapsto\, T_\gamma(z_0)
\varphi(\gamma) \sigma_G(g)$, where $\varphi$ is the
homomorphism in \eqref{vp-1}; here $\varphi(\gamma)$ is considered
as an element of $K$ using the natural identification of $\widehat{K}
\setminus K$ with $K$.

We will show that $\rho'_{x,\gamma}$ in \eqref{rpx} is independent
of the choice of $\gamma$. To prove this, take $\delta\, :=\,\gamma
\circ b\,=\, b\gamma$, where $b\, \in\, \pi_1(X,x)$. If $T_\delta\,:\,
(E_G)_x\, \longrightarrow\, (E_G)_{\sigma(x)}$ is the parallel
translation, with respect to $\nabla$, along $\delta$, then
$$
T_\delta (z_0)\varphi(\delta)\,=\,
(T_\gamma (z_0)\varphi(b^{-1}))(\varphi(b)\varphi(\gamma))
\,=\, T_\gamma(z_0) \varphi(\gamma)\, .
$$
This implies that $\rho'_{x,\gamma}\,=\, \rho'_{x,\delta}$.

Note that $\rho'_{x,\gamma}$ takes the natural action of $G$ on
$(E_G)_x$ to the action of $G$ on $(E_G)_{\sigma(x)}$ obtained by
twisting, using $\sigma_G$, the natural action of $G$ on 
$(E_G)_{\sigma(x)}$.

Now take any point $y\, \in\, X$. We will construct an 
isomorphism of $(E_G)_y$ with $(E_G)_{\sigma(y)}$.

Take any smooth path $\delta$ from $x$ to $y$. So
$\sigma(\delta)$ is a path from $\sigma(x)$ to
$\sigma(y)$. Let
$$
T_\delta\, :\, (E_G)_x\, \longrightarrow\, (E_G)_{y}
$$
and
$$
T_{\sigma(\delta)}\, :\, (E_G)_{\sigma(x)}\, \longrightarrow\, 
(E_G)_{\sigma(y)}
$$
be the parallel translations, with respect to the connection
$\nabla$, along $\delta$ and $\sigma(\delta)$ respectively.
Let
$$
\rho'_{y,\delta}\, :\, (E_G)_y\, \longrightarrow\, (E_G)_{\sigma(y)}
$$
be the map defined by $z\, \longmapsto\, T_{\sigma(\delta)}\circ
\rho'_{x,\gamma}\circ (T_\delta)^{-1}(z)$, where $\rho'_{x,\gamma}$ is 
constructed in \eqref{rpx} (we have shown that $\rho'_{x,\gamma}$ is 
independent of $\gamma$).

The above map $\rho'_{y,\delta}$ is again independent of the
choice of the path $\delta$. To prove this, take
$\eta\, :=\, \delta\circ b$, where $b\, \in\, \pi_1(X,x)$. Now
$$
T_{\sigma(\eta)}\circ\rho'_{x,\gamma}\circ (T_\eta)^{-1}
\,=\,T_{\sigma(\delta)} (T_{\sigma(\delta)})^{-1}\circ
T_{\sigma(\eta)}\circ T_\gamma \varphi(\gamma)\circ 
\varphi(b)(T_\delta)^{-1}
$$
$$
=\, T_{\sigma(\delta)}\circ
T_{\sigma(\delta^{-1}\circ\eta)\circ\gamma}\circ \varphi(\gamma b)
\circ (T_\delta)^{-1}\,=\, T_{\sigma(\delta)}\circ T_{\sigma(b)\circ
\gamma}\circ\varphi(\gamma b)\circ (T_\delta)^{-1}\, .
$$
Therefore, to prove that $\rho'_{y,\delta}\,=\, \rho'_{y,\eta}$,
it suffices to show that
\begin{equation}\label{r1}
T_{\sigma(b)\circ\gamma}\varphi(\gamma b)
\,=\, \rho'_{x,\gamma}\, .
\end{equation}
But $\gamma b$ is, by definition, $\sigma(b)\circ\gamma$. Therefore,
$$
T_{\sigma(b)\circ\gamma}\varphi(\gamma b)
\,=\,\rho'_{x,\sigma(b)\circ \gamma}\, ,
$$
where $\rho'_{x,\sigma(b)\circ \gamma}$ is constructed as
in \eqref{rpx}. But we have seen that $\rho'_{x,\gamma}$ is
independent of $\gamma$. Therefore, \eqref{r1} holds. Hence
$\rho'_{y,\delta}$ is independent of the choice of $\delta$.

Using the identification of the total spaces of $E_G$
and $\sigma^*\overline{E}_G$ (see Remark \ref{rem0}), the
above maps $\rho'_y\,
:=\, \rho'_{y,\delta}$, $y\, \in\, X$, together define a map
$$
\rho\, :\, E_G\, \longrightarrow\, \sigma^*\overline{E}_G\, .
$$
It can be checked that the pair $(E_G\, ,\rho)$ is a real
principal $G$--bundle. In view of Proposition \ref{prop2},
the real principal $G$--bundle $(E_G\, ,\rho)$ is polystable
because $E_G$ is polystable.
\end{proof}

It should be clarified that Theorem \ref{thm1} is not true
if the assumption that $G$ is semisimple is removed.

\section{Moduli space of principal bundles}

In this section it is assumed that $G$ is reductive. We also 
assume that $\text{genus}(X)\, \geq\, 3$.

Topological isomorphism classes of principal $G$--bundles
over $X$ are parametrized by the fundamental group $\pi_1(G)$;
any principal $G$--bundle is topologically trivial both on $X
\setminus\{x\}$ and on a neighborhood $D$ of $x$, and the resulting
map $D\setminus \{x\}\, \longrightarrow\, G$ produces an
element of $\pi_1(G)$. Fix a topological isomorphism class
$$
\lambda\, \in\, \pi_1(G)\, .
$$

Let ${\mathcal M}_X(G)$ be the moduli space of stable
principal $G$--bundles on $X$ of the given topological
type $\lambda$; see \cite{Ra2} for the
construction of ${\mathcal M}_X(G)$. This moduli space is
a normal quasiprojective complex variety; its dimension
is $\dim G(\text{genus}(X)-1)+ \dim Z$, where $Z$, as before,
is the center of $G$.

Let $E_G$ be a holomorphic principal $G$--bundle on $X$, and
let $Q$ be a parabolic subgroup of $G$. There is a natural
bijective correspondence between the holomorphic reductions
of structure group of $E_G$ to $Q$ and the holomorphic reductions
of structure group of $\sigma^*\overline{E}_G$ to the parabolic
subgroup $\sigma_G(Q)$. The identification between the total spaces
of $E_G$ and $\sigma^*\overline{E}_G$ (see Remark
\ref{rem0}) takes the total space of
a reduction of structure group of $E_G$ to $Q$ to the
total space of the corresponding reduction
of structure group of $\sigma^*\overline{E}_G$ to $\sigma_G(Q)$.
Using this correspondence between reductions
of structure group it follows immediately that
$E_G$ is stable if and only if $\sigma^*\overline{E}_G$ is stable.

We fix the topological isomorphism class $\lambda$ such that
$\sigma^*\overline{E}_G$ is topologically isomorphic to
$E_G$ for $E_G\, \in\, {\mathcal M}_X(G)$.
Let
\begin{equation}\label{mi}
\eta\, :\, {\mathcal M}_X(G)\, \longrightarrow\,
{\mathcal M}_X(G)
\end{equation}
be the anti-holomorphic involution defined by $E_G\, \longmapsto
\,\sigma^*\overline{E}_G$.

A holomorphic principal $G$--bundle $E_G$ is said to \textit{admit
a pseudo-real structure} if there is a holomorphic isomorphism of
principal $G$--bundles
$$
\rho\, :\, E_G\, \longrightarrow\,\sigma^*\overline{E}_G
$$
such that the pair $(E_G\, ,\rho)$ is pseudo-real.

Let
$$
{\mathcal M}^s_X(G)\, \subset\, {\mathcal M}_X(G)
$$
be the smooth locus of the variety. The involution $\eta$ in
\eqref{mi} preserves ${\mathcal M}^s_X(G)$.

\begin{theorem}\label{thm2}
Take any principal $G$--bundle $E_G\, \in\, {\mathcal M}^s_X
(G)$. Then this $E_G$ is fixed by the involution $\eta$ if and
only if $E_G$ admits a pseudo-real structure.
\end{theorem}

\begin{proof}
Let $(E_G\, ,\rho)$ be a pseudo-real principal $G$--bundle.
Then $\rho$ is an isomorphism of $E_G$ with
$\sigma^*\overline{E}_G$. Therefore, if $E_G\, \in\, {\mathcal 
M}^s_X(G)$, then $\eta(\{E_G\})\,=\, \{E_G\}$. We will now prove the
converse.

For a holomorphic principal $G$--bundle $F_G$, we have
$Z\, \subset\, \text{Aut}(F_G)$ (see
\eqref{cZ}). A stable principal $G$--bundle $F_G$
on $X$ is called \textit{regularly stable} if
$Z\, =\, \text{Aut}(F_G)$. Note that any stable principal
$\text{SL}(n, {\mathbb C})$--bundle is regularly stable.

The smooth locus ${\mathcal M}^s_X(G)$ coincides with the
locus in ${\mathcal M}_X(G)$ of regularly stable principal
$G$--bundles \cite[Corollary 3.4]{BH}; the assumption that
$g\, \geq\, 3$ is needed here (note that the moduli space
of principal $\text{SL}(2,{\mathbb C})$--bundles on a curve
of genus two is smooth as it is isomorphic to ${\mathbb C}
{\mathbb P}^3$).

Let $E_G\, \in\, {\mathcal M}_X(G)$ be a regularly stable
principal $G$--bundle such that $\eta(\{E_G\})\,=\, \{E_G\}$.
Fix a holomorphic isomorphism
$$
\theta\, :\, E_G\, \longrightarrow\, \sigma^*\overline{E}_G\, .
$$
Consider
$$
\sigma^*\overline{\theta}\, :\, \sigma^*\overline{E}_G
\, \longrightarrow\,
\sigma^*\overline{\sigma^*\overline{E}}_G\,=\,
\sigma^*\sigma^*\overline{\overline{E}}_G\, =\, E_G
$$
as in Definition \ref{def1}. The composition $(\sigma^*
\overline{\theta})\circ \theta$ is a holomorphic automorphism
of $E_G$. Since $E_G$ is regularly stable, there is an element
\begin{equation}\label{z00}
z_0\, \in\, Z
\end{equation}
such that $(\sigma^*\overline{\theta})\circ \theta$
coincides with the action of $z_0$ on $E_G$.

It is straightforward to check that $\theta$ commutes with the
composition $(\sigma^*\overline{\theta})\circ\theta$.
Indeed, in
terms of the identification of the total space of $E_G$ with
that of $\sigma^*\overline{E}_G$ (see Remark \ref{rem0}),
the map $\sigma^*\overline{\theta}$
coincides with $\theta$; this immediately 
implies that $(\sigma^*\overline{\theta})
\circ \theta$ commutes with $\theta$. Since $\theta$ commutes
with $(\sigma^*\overline{\theta})\circ\theta$, and
$\theta$ is an isomorphism
between $E_G$ and $\sigma^*\overline{E}_G$, from the construction
of $\overline{E}_G$ it follows immediately that
$$
\sigma_G(z_0)\, =\, z_0\, ,
$$
where $z_0$ is the element in \eqref{z00} (recall that
$\overline{E}_G$ is obtained from $E_G$ by twisting the action
of $G$ on $E_G$ by $\sigma_G$). Therefore,
$(E_G\, ,\theta)$ is pseudo-real. This completes the proof.
\end{proof}

\begin{remark}
{\rm The proof of Theorem \ref{thm2} shows that a regularly
stable principal $G$--bundle is fixed by the involution $\eta$ if
and only if $E_G$ admits a pseudo-real structure. Note that
this statement is valid even even if the genus of $X$ is two.}
\end{remark}

\section*{Acknowledgements}

We are very grateful to the referee for detailed comments.
This work was carried out while both the authors were visiting
the Issac Newton Institute. We thank Issac Newton Institute
for its hospitality.


\end{document}